\documentclass[11pt]{amsart}
\usepackage{amsmath} 
\usepackage{amsthm}
\usepackage{amssymb} 
\usepackage{amsfonts}
\usepackage{graphicx}
\usepackage{color}
\usepackage[normalem]{ulem}
\usepackage{comment}

\usepackage{caption,subcaption,pinlabel}
\usepackage{graphics}

\newtheorem{theorem}{Theorem}[section]
\newtheorem{lemma}[theorem]{Lemma}
\newtheorem{proposition}[theorem]{Proposition}

\newtheorem{question}[theorem]{Question}

\newtheorem{corollary}[theorem]{Corollary}

\theoremstyle{definition}
\newtheorem{definition}[theorem]{Definition}
\newtheorem{example}[theorem]{Example}

\newcommand{\Z}{\mathbb{Z}}

\newcommand{\T}{\mathcal{T}}
\newcommand{\M}{\mathcal{M}}
\let\int\relax
\newcommand{\int}{\mathring}

\newcommand{\boundary}{\partial}
\newcommand{\into}{\hookrightarrow}

\usepackage{accents}

\DeclareMathSymbol{\wtilde}{\mathord}{largesymbols}{"65}

    \author[Peter Lambert-Cole and Maggie Miller]{Peter Lambert-Cole and Maggie Miller}
    \title[Trisections of $5$-manifolds]{Trisections of $5$-manifolds}
    
       \address{Georgia Institute of Technology, Atlanta, GA 30332, USA}\email{plc@math.gatech.edu}
    \address{Princeton University, Princeton, NJ 08540, USA}\email{maggiem@math.princeton.edu}

\begin{document}

  \maketitle

  \begin{abstract}
Gay and Kirby introduced the notion of a trisection of a smooth 4-manifold, which is a decomposition of the 4-manifold into three elementary pieces.  Rubinstein and Tillmann later extended this idea to construct multisections of piecewise-linear (PL) manifolds in all dimensions.  Given a PL manifold $Y$ of dimension $n$, this is a decomposition of $Y$ into $\lfloor \frac{n}{2} \rfloor + 1$ PL submanifolds.  We show that every smooth, oriented, compact 5-manifold admits a smooth trisection compatible with any desired trisection of its boundary.  
  \end{abstract}

  \section{Introduction}
  In this paper, we study trisections of compact $5$-manifolds. First, we review the concept of a trisection of a (closed) $4$-manifold.
 \begin{definition}[{\cite{gaykirby}}]\label{tridef4}
A $(g;k_1,k_2,k_3)$ {\it trisection} of a smooth, oriented, closed 4-manifold $X$ is a triple $(X_1, X_2, X_3)$ such that
\begin{itemize}
\item $X=X_1\cup X_2\cup X_3$ and $X_i\cap X_j=\boundary X_i\cap\boundary X_j$ for each $i\neq j$,
\item Each $X_{i}\cong\natural_{k_i} S^1\times B^3$ is a 4-dimensional 1-handlebody,
\item Each double intersection $X_{i} \cap X_{j}$ is a 3-dimensional 1-handlebody, and
\item The triple intersection $\Sigma = X_1 \cap X_2 \cap X_3$ is a closed, oriented surface of genus-$g$.
\end{itemize}
\end{definition}

Moreover, every inclusion $X_i\into X$, $X_i\cap X_j\into X$, $\Sigma\into X$ is smooth.
Gay and Kirby proved that every closed, oriented smooth 4-manifold admits a trisection (which is unique up to a stabilization operation).  The genesis was their study of Morse 2-functions, although they also showed that it is possible to build a trisection from a handle decomposition (see the proof of Theorem~\ref{thrm:GK-tri}).  Subsequently, Rubinstein and Tillmann generalized these ideas and found nice decompositions of piecewise-linear manifolds in all dimensions~\cite{multisections}.  Given a PL manifold $Y$ of dimension $n$, this is a decomposition of $Y$ into $k = \lfloor \frac{n}{2} \rfloor +1$ PL submanifolds, each of which is an $n$-dimensional 1-handlebody.  The intersection of any $j$ pieces, for $j = 2,\dots,k$, must satisfy further restrictions on their topology.  When $n = 5$, the number of pieces is $\lfloor \frac{5}{2} \rfloor +1 = 3$ and so PL 5-manifolds admit trisections as well.

\begin{definition}\label{tridef5}
A {\emph{smooth trisection}} $\M=(Y_1, Y_2, Y_3)$ of a $5$-manifold $Y$ is a decomposition of $Y$ into three pieces $Y_1, Y_2, Y_3$ so that the following conditions hold:
\begin{itemize}
\item $Y=Y_1\cup Y_2\cup Y_3$, where $Y_i\cap Y_j=\boundary Y_i\cap\boundary Y_j$ for $i\neq j$.
\item Each $Y_i$ is smoothly embedded into $Y$.
\item For integers $k_1, k_2, k_3\ge 0$, we have $Y_i\cong\natural_{k_i} S^1\times B^4$.
\item Each $Y_i\cap Y_j$ is a $4$-manifold smoothly embedded into $Y$. Moreover, $Y_i\cap Y_j$ is a regular neighborhood of its $2$-skeleton.
\item The triple intersection $Y_i\cap Y_j\cap Y_k$ is a $3$-manifold smoothly and properly embedded in $Y$.
\item If $\boundary Y\neq\emptyset$, then the triple $(Y_1\cap\boundary Y, Y_2\cap\boundary Y, Y_3\cap\boundary Y)$ is a trisection of $\boundary Y$.
\end{itemize}
We refer to $Y_1\cap Y_2\cap Y_3$ as the {\emph{central submanifold}} of $\M$.
\end{definition}

Definition~\ref{tridef5} agrees completely with the definition of~\cite{multisections} (in the $5$-dimensional case), except where we require all inclusions to be smooth rather than piecewise linear.

From now on, ``trisection'' will always mean ``smooth trisection.''  Our main theorems are the following:

\begin{theorem}
\label{thrm:tri-closed}
Every closed, smooth, oriented 5-manifold $Y$ admits a trisection naturally induced by a chosen handle structure on $Y$.
\end{theorem}

\begin{theorem}\label{thrm:cobord}
Let $Y$ be smooth, oriented 5-manifold with positive boundary $A$ and negative boundary $B$.  Fix trisections $\T_A$ and $\T_B$ of $A$ and $B$, respectively.  There exist a trisection $\T_Y$ of $Y$ whose restriction to $A$ (resp. $B$) is $\T_A$ (resp. $\T_B$).
\end{theorem}

Theorem~\ref{thrm:cobord} together with the fact that every closed $4$-manifold admits a trisection~\cite{gaykirby} implies the following theorem.

\begin{theorem}
Every compact, smooth, oriented 5-manifold $Y$ admits a trisection naturally induced by a chosen handle structure on $Y$.
\end{theorem}

Similarly, Theorem~\ref{thrm:tri-closed} could be taken to be a consequence of Theorem~\ref{thrm:cobord}.

Note that we do not consider the question of uniqueness of smooth trisections of $5$-manifolds up to stabilization. In dimension $4$, the central submanifold of a trisection is an orientable surface, and the stabilization operation increases the genus of this surface. Since any two orientable surfaces are related by such stabilization, it is initially plausible that two trisections of a $4$-manifold are related by stabilization. In contrast, a trisection of a $5$-manifold has a $3$-manifold as its central submanifold. The natural stabilization operation on the trisection adds a connect-summand of $S^1\times S^2$ to this $3$-manifold. In general, two $3$-manifolds are not related by such stabilization, so we do not expect two arbitrary trisections of a $5$-manifold to be related by stabilization.

\begin{example}
Fix coordinates $(r,\theta,x,y,z)$ on $\mathbb{R}^5$, where $(r,\theta)$ are polar coordinates and $(x,y,z)$ are Cartesian. View $S^5$ as $\mathbb{R}^5\cup\{\infty\}$. For $i=1,2,3$, let $Y_i=\{2\pi(i-1)/3\le\theta\le 2\pi i/3\}\cup\{\infty\}$. Then $\T=(Y_1, Y_2, Y_3)$ is a trisection of $S^5$.

On the other hand, we may view $S^5$ as $\boundary(D^2\times D^2\times D^2)$. In this coordinate system, let $W_1=S^1\times D^2\times D^2$, $W_2=D^2\times S^1\times D^2$, $W_3=D^2\times D^2\times S^1$. Then $\T'=(W_1, W_2, W_3)$ is a trisection of $S^5$.

The central submanifold of $\T$ is a a $3$-sphere, while the central submanifold of $\T'$ is a $3$-torus. Since $S^3\#_m(S^1\times S^2)\not\cong T^3\#_n(S^1\times S^2)$ for any $m,n$, we conclude that $\T$ and $\T'$ have no common stabilization.
\end{example}

\begin{question}\label{uniquenessquestion}
Is there a suitable class of trisections of closed $5$-manifolds and a ``natural'' set of stabilization operations which relate any two of these trisections which are of the same $5$-manifold?
\end{question}

\subsection*{Acknowledgements}
The work in this paper was initiated and mostly completed at 
{\emph{Topology of Manifolds: interactions between high and low dimensions}} at MATRIX in January 2019. We thank Boris Lishak and Stephan Tillmann for interesting conversations at MATRIX and afterward about multisections. The second author's main takeaway was that trisections of closed 5-manifolds are significantly more complicated than those of closed 4-manifolds, and quadrisections of 6-manifolds are exponentially more so.

Thanks also to Mark Powell for a helpful comment.

The second author is a fellow in the National Science Foundation Graduate Research Fellowship program, under Grant No. DGE-1656466.

\section{Trisecting closed $5$-manifolds}

\subsection{Trisections of closed $4$-manifolds}

As preparation for the proof of Theorem~\ref{thrm:tri-closed}, we review the construction in~\cite{gaykirby} of a trisection from a handle decomposition of a closed 4-manifold $X$.  Roughly speaking, we partition the handles of $X$ into three subsets --- (1) the 0- and 1-handles; (2) the 2-handles; and (3) the 3-and 4-handles --- and each group becomes one sector of the trisection.

\begin{theorem}[\cite{gaykirby}]
\label{thrm:GK-tri}
Every closed, oriented, smooth 4-manifold admits a trisection.
\end{theorem}

\begin{proof}
Take a self-indexing Morse function $f$ on $X$ and let $k_i$ be the the number of index-$i$ critical points.  Without loss of generality, assume $k_0 = k_4 = 1$.  Identify the attaching link $L$ of the 2-handles in the level set $f^{-1}(3/2)$ and choose a tubular neighborhood $\nu(L)$ in $f^{-1}(3/2)$.  Choose a relative handle decomposition on the link complement $f^{-1}(3/2) \smallsetminus \nu(L)$, which we can assume consists of 1-,2- and 3-handles.  Let $H_1$ be the union of the 2- and 3-handles of this decomposition and let $H_2$ be the union of $\nu(L)$ with the 1-handles.  Then clearly $H_1$ and $H_2$ are 3-dimensional 1-handlebodies, meeting along a closed surface $S$.  Equivalently, this gives a Heegaard splitting $f^{-1}(3/2)= H_1 \cup_S H_2$ of the level set.

The attaching link lies completely in $H_2$, so by flowing along a gradient vector field we can find the cylinder $H_1 \times [3/2,5/2]$ in $X$.  Define $X_1 = f^{-1}([0,3/2]) \cup H_1 \times [3/2,2]$; it retracts onto the sublevel set $f^{-1}([0,3/2])$ and so is a 1-handlebody.  Similarly, define $X_3 =  f^{-1}([5/2,4]) \cup H_1 \times [3/2,2]$; it retracts on the superlevel set $f^{-1}([5/2,4])$ and is also a 1-handlebody.  Finally, let $X_2$ be the complement of $X_1 \cup X_3$ in $X$.  Abstractly, it is diffeomorphic to $H_2 \times I \cup \{\text{2-handles}\}$.  The manifold $H_2 \times I$ is a 1-handlebody and $H_2$ was obtained from $\nu(L)$ by attaching 1-handles.  Thus each 2-handle cancels a unique 1-handle in $H_2 \times I$.  Thus, the result is a 1-handlebody. Moreover, the double intersections $X_1\cap X_2=H_2$, $X_3\cap X_1= H_1$,$ X_2\cap X_3=(H_2$ Dehn surgered along $L)$ are all $3$-dimensional $1$-handlebodies. The central submanifold is the surface $X_1\cap X_2\cap X_3=S$.
\end{proof}

\subsection{Trisections of closed $5$-manifolds}

We can now describe how to obtain a trisection of a closed 5-manifold from a handle decomposition.  The essential idea, as in the prequel, is to partition the handles into three sets: (1) the 0- and 1-handles; (2) the 2- and 3-handles; and (3) the 4- and 5-handles.  

\begin{proof}[Proof of Theorem~\ref{thrm:tri-closed}]
Take a self-indexing Morse function $f$ and let $k_i$ be the the number of index-$i$ critical points.  Without loss of generality, assume $k_0 = k_5 = 1$.  In the level set $f^{-1}(5/2)$, let $S$ denote the attaching 2-spheres of the 3-handles above and let $R$ denote the belt 2-spheres of the 2-handles below.  We can assume they intersect transversely and then choose a tubular neighborhood $\nu(R \cup S)$.  Choose a relative handle decomposition of $f^{-1}(5/2)\smallsetminus\nu(R\cup S)$ which we can assume has no 0-handles.  Let $H_1$ be the union of the 2-,3- and 4-handles of this handle decomposition and let $H_2$ be the union of $\nu(R \cup S)$ with the 1-handles.  Clearly, $H_1$ and $H_2$ can be built with only 0-,1- and 2-handles and meet along a closed 3-manifold.

By flowing along a gradient vector field, we can find the cylinder $H_1 \times [3/2,7/2]$ in $Y$.  Define $Y_1 = f^{-1}([0,3/2]) \cup H_1 \times [3/2,5/2]$; it retracts onto the sublevel set $f^{-1}([0,3/2])$ and so is a 1-handlebody.  Similarly, define $Y_3 =  f^{-1}([7/2,5]) \cup H_1 \times [5/2,7/2]$; it retracts on the superlevel set $f^{-1}([7/2,5])$ and is also a 1-handlebody.  Finally, let $Y_2$ be the complement of $Y_1 \cup Y_3$ in $Y$.  Although it contains the 2- and 3-handles of $Y$, it is abstractly diffeomorphic to the union of $H_2 \times [0,1]$ with two collections of 3-handles.  The 3-handles of $Y$ are attached along the link $S \subset H_2 \times \{1\}$.  By turning the 2-handles of $Y$ upside down, we can view them as 3-handles attaching along $R \subset H_2 \times \{0\}$.  Each of these three handles cancel a unique 2-handle in $H_2 \times [0,1]$.  Moreover, these are the only 2-handles and so the result is a 1-handlebody. Moreover, the double intersections $Y_1\cap Y_2=(H_2$ surgered along the belt spheres of the $2$-handles$)$, $Y_3\cap Y_1= H_1$, $Y_2\cap Y_3=(H_2$ surgered along the attaching spheres of the $3$-handles$)$ are all $4$-dimensional $0$, $1$, $2$-handlebodies. The central submanifold is the $3$-manifold $Y_1\cap Y_2\cap Y_3=\boundary H_1$.
\end{proof}

\section{Trisecting $5$-manifolds with boundary}

Tillmann and Rubinstein~\cite{multisections} do not fix a definition of a multisection of a manifold with boundary. A relative trisection of a $4$-manifold $X$ with boundary is well understood, having been originally introduced in~\cite{gaykirby} and fleshed out in~\cite{nickthesis}. A diagrammatic theory for relative trisections then appeared in~\cite{reldiagrams}, and has continued to appear throughout trisection literature.  Briefly, a relative trisection of a $4$-manifold with boundary is required to induce an open book on $\boundary X$, so that relative trisections inducing the same boundary data can be glued to find trisections of the union. We give an analogous condition in this section.


\subsection{Gluing cobordisms}

In order to build a trisection of $Y$ from trisections of elementary pieces, we need to check that the topological conditions on a trisection hold after gluing together a pair of trisected cobordisms.

Let $Y$ be a compact, smooth 5-manifold with boundary. Let $\M = (Y_1,Y_2,Y_3)$ be a trisection of $Y$, and let $X$ be one of the boundary components of $Y$.  A trisection $\M$ of $Y$ is {\it compatible} with a trisection $\T = (X_1,X_2,X_3)$ of $X$ if its restriction $\M|_X$ is equal to $\T$.  A trisection $\M$ of $Y$ is {\it strongly compatible} with $\T$ if it is compatible with $\T$ and the inclusion $X_i \hookrightarrow Y_i$ maps a core of the $4$-dimensional handlebody $X_i$ to a core of the $5$-dimensional handlebody $Y_i$ for each $i$. 
If $\M$ is strongly compatible with its restriction to $X$, we also say $\M$ is strongly compatible with $X$.  If $\M$ is strongly compatible with every boundary component, we say that $\M$ is strongly compatible with $Y$.

Let $Y$ be a 5-manifold with boundary.  To view $Y$ as a cobordism of $4$-manifolds, choose a decomposition $\partial Y = \partial Y_+ \coprod (-\partial Y_-)$ (where one of $\partial Y_\pm$ may be empty).  Suppose that $W$ is another cobordism of $4$-manifolds and there is a diffeomorphism $\phi: \partial Y_+ \rightarrow \partial W_-$.  Then we can {\it glue} $Y$ to $W$ and obtain a new cobordism $Y \cup_{\phi} W$ with boundary $\partial Y \cup_{\phi} \partial W = \partial W_+ \coprod (- \partial Y_-)$.

\begin{lemma}
\label{gluinglemma}
Let $Y$ and $W$ be cobordisms of $4$-manifolds and let $\phi: \partial Y_+ \rightarrow \partial W_-$ be a diffeomorphism.  Suppose that $\M_Y = (Y_1,Y_2,Y_3)$ and $\M_W = (W_1,W_2,W_3)$ are strongly compatible trisections of $Y$ and $W$ (respectively) and that $\phi$ identifies $\M|_{\partial Y_+}$ with $\M|_{\partial W_-}$.  Then $\M_{Y \cup W} = (Y_1 \cup W_1,Y_2 \cup W_2, Y_3 \cup W_3)$ is a trisection that is strongly compatible with $Y \cup_{\phi} W$.
\end{lemma}

\begin{proof}
By definition, each of $Y_i$ and $W_i$ has a handle decomposition with only $0$- and $1$-handles. Since they are glued along a $1$-handlebody, $Y_i\cup W_i$ has a handle decomposition with only $0$-, $1$-, and $2$-handles. The $2$-handles may be chosen to each run geometrically along a $1$-handle of $Y_i\cap\partial Y_+$ and an identified $1$-handle of $W_i\cap \partial W_-$ (as well as other $1$-handles) once, since $\M_Y$ and $\M_W$ are strongly compatible with $Y$ and $W$ respectively. By assumption, the $2$-handles can then be cancelled geometrically, so we conclude that $Y_i\cup W_i\cong\natural S^1\times B^4$.

For $i\neq j$, we have $(Y_i\cup W_i)\cap(Y_j\cup W_j)=(Y_i\cap Y_j)\cup_{Y_i\cap Y_j\cap (\partial Y_+\cong\partial W_-)} (W_i\cap W_j)$. Therefore, $(Y_i\cup W_i)\cap(Y_j\cup W_j)$ is obtained by gluing two $4$-dimensional $0$-,$1$-,$2$-handlebodies along a $3$-dimensional handlebody. To glue along a handlebody, we need need only add $1$- and $2$-handles, so $(Y_i\cup W_i)\cap(Y_j\cup W_j)$ is a $0$-, $1$-, $2$-handlebody, as desired.

The rest of Definition~\ref{tridef5} follows easily.

\end{proof}

\subsection{Standard trisections}

Our local models of trisections are obtained by pulling back a trisection on the unit disk $D$ in $\mathbb{R}^2$.  In radial coordinates, the symmetric trisection $D = D^s_1 \cup D^s_2 \cup D^s_3$ is defined by choosing the following subsets:
\begin{align*}
D^s_1 &= \left\{ 0 \leq \theta \leq \frac{2\pi}{3} \right\}, & D^s_2 &= \left\{\frac{2\pi}{3} \leq \theta \leq \frac{4\pi}{3} \right\}, &  D^s_3 &= \left\{ \frac{4\pi}{3} \leq \theta \leq 2 \pi \right\}.
\end{align*}
The symmetric trisection is symmetric under rotation by $2\pi/3$ (up to permuting indices). We also define a rectangular trisection $D = Y_1 \cup Y_2 \cup Y_3$ by setting
\begin{align*}
D^r_1 &= \{ x \geq 0 \}, & D^r_2 &= \{ x \leq 0, y \geq 0\}, & D^r_3 &= \{ x \leq 0, y \leq 0\}.
\end{align*}
Geometrically, the rectangular trisection is asymmetric. Up to diffeomorphism, this trisection is equivalent to the symmetric trisection.

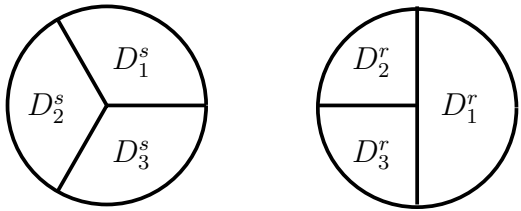
\begin{figure}[h!]
\centering
\labellist
	\large\hair 2pt
	\pinlabel $D^s_1$ at 100 110
	\pinlabel $D^s_2$ at 45 80
	\pinlabel $D^s_3$ at 100 50
	\pinlabel $D^r_1$ at 312 80
	\pinlabel $D^r_2$ at 255 110
	\pinlabel $D^r_3$ at 255 50
\endlabellist
\includegraphics[width=.6\textwidth]{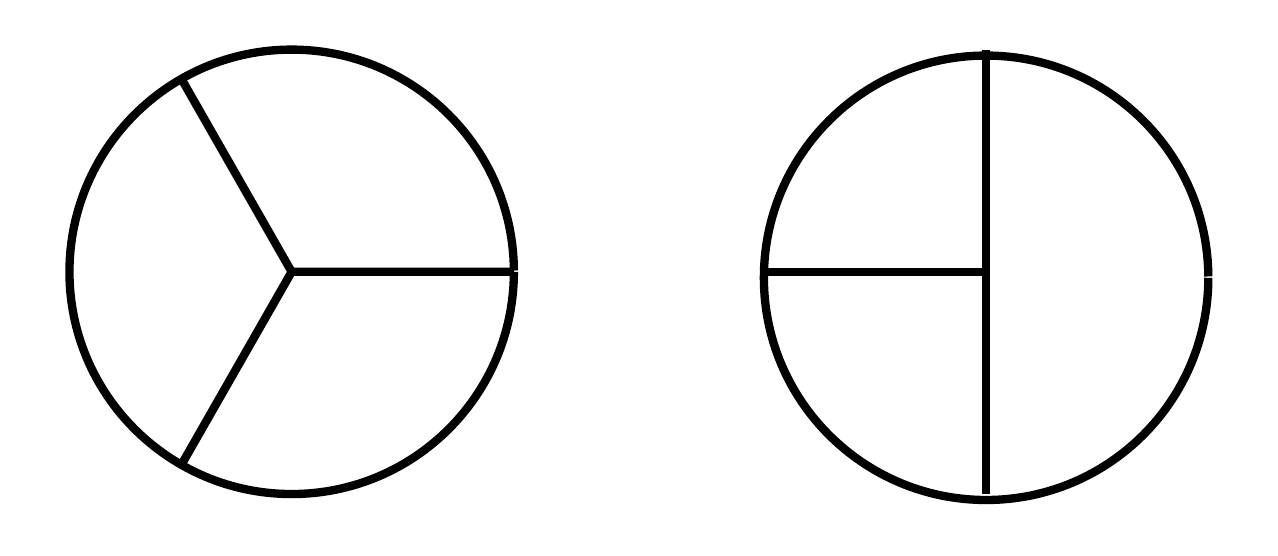}
\caption{The symmetric ({\it left}) and rectangular ({\it right}) trisections of the unit disk in $\mathbb{R}^2$.}
\label{fig:disk-tri}
\end{figure}

\begin{definition}
The {\it standard trisection} of $S^{k}$ for $k\ge 2$ 
is the decomposition $\T_{std} = \{\pi^{-1}(D^s_i) \cap S^{k}\}$ where $\pi: \mathbb{R}^{k+1} \rightarrow \mathbb{R}^2$ is a coordinate projection and $D^s_1 \cup D^s_2 \cup D^s_3$ is the standard trisection of the unit disk in $\mathbb{R}^2$.

The {\it standard trisection} of $B^{k}$, for $k\ge 2$ 
is the decomposition $\T_{std} = \{\pi^{-1}(D^s_i) \cap B^{k}\}$ where $\pi: \mathbb{R}^k \rightarrow \mathbb{R}^2$ is a coordinate projection and $D^s_1 \cup D^s_2 \cup D^s_3$ is the standard trisection of the unit disk in $\mathbb{R}^2$.
\end{definition}

When $k\not\in\{4,5\}$, a multisection of $S^k$ is not a trisection.

Note that we could have defined the standard trisection using the rectangular trisection of the disk rather than the symmetric trisection of the disk. Both definitions give equivalent (up to isotopy) trisections. Generally, we will use the coordinates of the rectangular trisection instead (for convenience). We will specify ``standard symmetric trisection'' or ``standard rectangular trisection,'' but a reader comfortable with trisections may read this as ``standard trisection.''

\begin{lemma}
\label{lemma:std}
Let $\T_{std} = (S_1,S_2,S_3)$ be the standard trisection of $S^{k}$.  Then each $S_i$ is diffeomorphic to $B^{k}$; each double intersection $S_i \cap S_j$ is diffeomorphic to $B^{k-1}$ and the central surface is diffeomorphic to $S^{k-2}$.

Let $\T_{std} = (D_1,D_2,D_3)$ be the standard trisection of $B^{k}$.  Then each $D_i$ is diffeomorphic to $B^{k}$; each double intersection $D_i \cap D_j$ is diffeomorphic to $B^{k-1}$ and the central surface is diffeomorphic to $B^{k-2}$.
\end{lemma}

\begin{proposition}
Let $\T_{std}$ be the standard trisection of $S^4$ and $\M_{std}$ be the standard trisection of $B^5$.
\begin{enumerate}
\item $\T_{std}$ is a trisection of $S^4$,
\item $\M_{std}$ is a trisection of $B^5$, and
\item $\M_{std}$ is strongly compatible with $\T_{std}$.
\end{enumerate}
\end{proposition}

The {\it standard trisection} of $S^1 \times S^3$ is obtained from the standard trisection of $S^3$ by taking the product of each sector with $S^1$.  It is clear from Lemma~\ref{lemma:std} that this is a trisection.

\subsection{Trisection stabilization}

\begin{definition}
Let $\T=(X_1,X_2,X_3)$ be a $(g;k_1,k_2,k_3)$-trisection of a $4$-manifold $X$. We say that a $(g+1;k_1+1,k_2,k_3)$-trisection $\T'$ of $X$ is {\emph{an elementary $1$-stabilization of $\T$}} if there exists a boundary-parallel arc $C$ properly embedded in $X_2\cap X_3$ so that
\begin{align*}
X'_1&=X_1\cup\overline{\nu(C)},\\
X'_2&=X_2\setminus\nu(C),\\
X'_3&=X_3\setminus\nu(C),
\end{align*}
for some fixed open neighborhood $\nu(C)$ of $C$. We say that $C$ is the {\emph{stabilization arc}} of the stabilization $\T\mapsto \T'$.

We similarly define elementary $2$- and $3$-stabilization by permuting the indices $1,2,$ and $3$.

\end{definition}

\begin{lemma}\label{stabilize}
Let $X$ be a closed $4$-manifold. Fix a $(g;k_1,k_2,k_3)$-trisection $\T=(X_1,X_2,X_3)$ of $X$. Let $\T'$ be an elementary $1$-stabilization of $\T$, so that $\T'=(X'_1,X'_2,X'_3)$ is a $(g+1;k_1+1,k_2,k_3)$-trisection of $X$.  There exists a smooth multisection $\M=(Y_1,Y_2,Y_3)$ of $X \times I$ whose restriction to $X \times \{0\}$ is $\T$ and whose restriction to $X \times \{1\}$ is $\T'$.
\end{lemma}

\begin{proof}
See Figure~\ref{fig:stabilization} for a schematic.

Let $C$ be the stabilization arc of $\T\mapsto\T'$. For $i=1,2,3$, let $Y_i:=(X_i\times[0,1/2])\cup(X'_i\times[1/2,1])$, so that $Y=Y_1\cup Y_2\cup Y_3$. We immediately have $Y_i\cap(X\times 0)=X_i$ and $Y_i\cap(X\times 1)=X_1$. Oviously this implies that $\M=(Y_1,Y_2,Y_3)$ induces a trisection on $\boundary Y$.

Note that for each $i$ and $j$, $Y_i$ and $Y_i\cap Y_j$ strongly deformation retract onto $Y_i\cap (X\times 1/2)$ and $(Y_i\cap Y_j)\cap (X\times 1/2)$, respectively. We will describe each of these intersections.

\begin{align*}
Y_1\cap (X\times 1/2)&=X'_1\times 1/2,\\
Y_2\cap (X\times 1/2)&=X_2\times 1/2,\\
Y_3\cap (X\times 1/2)&=X_3\times 1/2.\\
\end{align*}

We conclude that $Y_1\cong \natural_{k_1+1}S^1\times B^4$,  $Y_2\cong \natural_{k_2}S^1\times B^4$, and  $Y_1\cong \natural_{k_3}S^1\times B^4$. Moreover,

\begin{align*}
Y_1\cap Y_2\cap (X\times 1/2)&=((X_1 \cap X_2)\cup(X_2\cap\overline{\nu(C)}))\times 1/2,\\
Y_2\cap Y_3\cap (X\times 1/2)&=(X_2\cap X_3)\times1/2,\\
Y_3\cap Y_1\cap (X\times 1/2)&=((X_1 \cap X_3)\cup(X_3\cap\overline{\nu(C)}))\times 1/2.\\
\end{align*}

Then immediately, $Y_2\cap Y_3\cong\natural_g S^1\times B^3$. Moreover, we note that $Y_1\cap Y_2\cap( X\times 1/2)$ is obtained from the $3$-dimensional handlebody $(X_1\cap X_2)\times 1/2$ by attaching a $4$-dimensional $1$-handle, so strongly deformation retracts to a $1$-skeleton. Therefore, $Y_1\cap Y_2\cong\natural_{g+1} S^1\times B^3$. Similarly, $Y_3\cap Y_1\cong\natural_{g+1} S^1\times B^3$.

Finally, we have that $Y_1\cap Y_2\cap Y_3$ is the $3$-dimensional trace of the cobordism from $X_1\cap X_2\cap X_3$ to $X'_1\cap X'_2\cap X'_3$ obtained by attaching the $3$-dimensional $1$-handle $\overline{(\nu(C)\smallsetminus X_1)}$. That is, $Y_1\cap Y_2\cap Y_3$ is a compression body from a genus $(g+1)$-surface to a genus $g$-surface. 

\end{proof}

We will refer to the trisected $5$-manifold $Y$ of Lemma~\ref{stabilize} as a stabilization cobordism. Figure~\ref{fig:stabilization} shows a schematic of a stabilization cobordism.

\begin{figure}
\includegraphics[width=.7\textwidth]{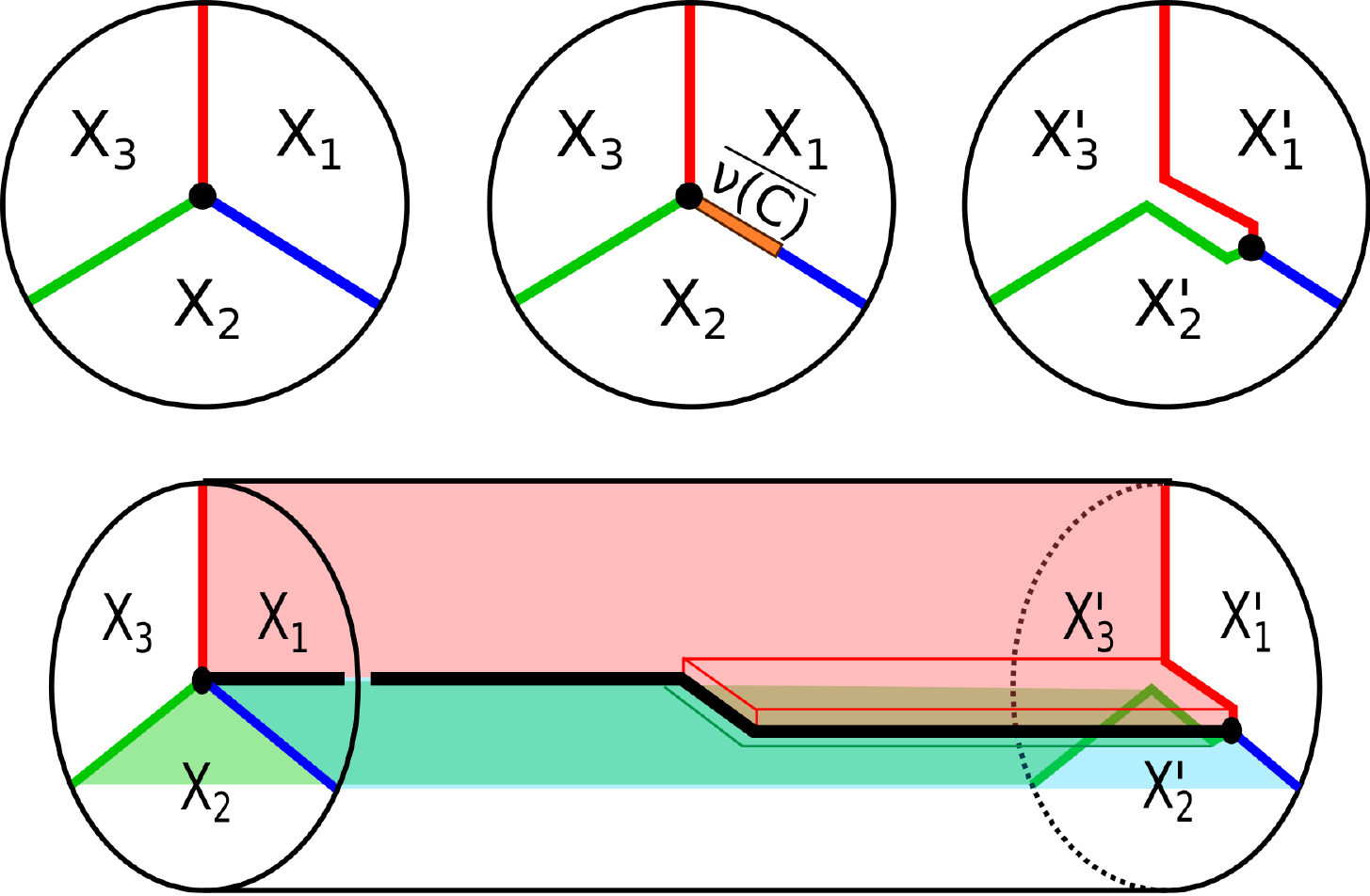}
\caption{A schematic of a stabilization cobordism corresponding to an elementary $1$-stabilization about arc $C\subset X_2\cap X_3$. This is a cobordism from $X$ to $X$ inducing trisection $\T$ on the left boundary and $\T'$ on the right boundary (where $\T'$ is obtained from $\T$ by $1$-stabilization), as in Lemma~\ref{stabilize}. }\label{fig:stabilization}
\end{figure}

%
%
%

\begin{proposition}
\label{prop:change-trisection}
Let $X$ be a closed, oriented, smooth 4-manifold and let $\T_0,\T_1$ be trisections of $X$.  There exists a trisection $\M=(Y_1, Y_2, Y_3)$ of $X \times [0,1]$ whose restriction to $X \times \{i\}$ is $\T_i$.
\end{proposition}

\begin{proof}
By~\cite[Theorem 11]{gaykirby}, there exists a common stabilization $\widetilde{\T}$ of $\T_0$ and $\T_1$. Therefore, the claim holds by induction on Lemmas~\ref{stabilize} and~\ref{gluinglemma}
\end{proof}

\subsection{Morse theory for manifolds with boundary}

For a comprehensive treatment, we refer the reader to~\cite{BNR}.

Let $f$ be a Morse function on $X$; for the sake of exposition we assume the critical values are all distinct.  By the Morse lemma, we can choose coordinates around every nondegenerate critical point of $f$ in which the function takes the form
\begin{equation}
\label{eq:Morse}
f(x_1,\dots,x_n) = x_1^2 + \dots + x_{n - k}^2 - x_{n - k +1}^2 - \dots - x_{n}^2
\end{equation}
for some $k$, which is called the {\it index} of the critical point.  Let $X_s = f^{-1}((\infty,s])$.  Up to diffeomorphism, the sublevel set $X_{\epsilon}$ can be obtained from the sublevel set $X_{-\epsilon}$ by attaching a $k$-handle along some $S^{k-1}$ in the level set $f^{-1}(-\epsilon)$.

Now suppose that $X$ has nonempty boundary and $f$ is a Morse function on $X$ that restricts to a Morse function on $\partial X$.  If $z \in \partial X$ is a critical point of $f$, we can find Morse coordinates near $z$ as in Equation~\ref{eq:Morse} and such that $\partial X$ is sent to the hyperplane $\{x_j = 0\}$ for some $j$.  The critical point $z$ is {\it boundary unstable} if $1 \leq j \leq n - k$ and is {\it boundary stable} if $n - k + 1 \leq j \leq n$.  

The topological change to a sublevel set when $f$ has a Morse critical point on the boundary depends on whether the critical point is boundary stable or boundary unstable.

\begin{proposition}
\label{prop:Morse-boundary}
Let $z \in \partial X$ be a Morse critical point of index-$k$.
\begin{enumerate}
\item If $z$ is boundary stable, then $X_{\epsilon}$ is diffeomorphic to $X_{-\epsilon}$.  Furthermore, $(\partial X)_{\epsilon}$ is obtained from $(\partial X)_{-\epsilon}$ by attaching a handle of index-$k - 1$.
\item If $z$ is boundary unstable, then $X_{\epsilon}$ is obtained from $X_{-\epsilon}$ by attaching a handle of index-$k$.  Furthermore, $(\partial X)_{\epsilon}$ is obtained from $(\partial X)_{-\epsilon}$ by attaching a handle of index-$k$.
\end{enumerate}
\end{proposition}

\begin{proof}
The statements about the topology of $\partial X$ are standard Morse theory.  The statements about the topology of $X$ are the combination of Lemmas 2.18 and 2.19 and Theorem 2.27 in~\cite{BNR}.
\end{proof}

\subsection{Index-1}

In this and the following subsections, let $Y$ be a cobordism between closed $4$-manifolds $X$ and $Z$; let $\T_X=(X_1,X_2,X_3)$ and $\T_Z=(Z_1, Z_2, Z_3)$ be trisections of $X$ and $Z$, respectively; and let $f:Y\to [0,1]$ be a relative Morse function so that $\boundary Y=f^{-1}(\{0,1\})$, where $X=f^{-1}(0)$ and $Z=f^{-1}(1)$. 

Suppose that $S$ is an embedded $S^0$ in the central surface $\Sigma$ of the trisection $\T_X$.  Let $\nu_{\Sigma}(S)$ be a tubular neighborhood  in the central surface.  Then we can choose a tubular neighborhood $\nu_X(S) \cong \nu_{\Sigma}(S) \times D^2$ such that $\T_X$ restricts to a trisection obtained by pulling back the standard (which we view to be the rectangular) trisection of the disk under the projection $\nu_X(S) \rightarrow D^2$.  Note that the trisection determines a framing of the bundle $\nu_{\Sigma}(S) \times D^2$, but this framing is unique up to isotopy since $\nu_{\Sigma}(S) \cong S^0 \times D^2$.

\begin{proposition}
\label{prop:index1}
Suppose that there is a unique  critical point of $f$ of index-$1$ in the interior of $Y$. There exists a trisection $\M=(Y_1, Y_2, Y_3)$ of $Y$ that is strongly compatible with the trisections $\T_X$ and $\T_Z$.
\end{proposition}

\begin{proof}
First we describe the local model in Morse coordinates.  Near a Morse critical point of index 1, we have Morse coordinates such that
\[f = x_1^2 + x_2^2 + x_3^2 + x_4^2 - x_5^2\]
We can view this as a function on $\mathbb{R}^2 \times \mathbb{R}^3$ and decompose $f$ as $g + \widetilde{f}$, where
\begin{align*}
g(x_1,x_2) &= x_1^2 + x_2^2 & \widetilde{f}(x_3,x_4,x_5) &= x_3^2 + x_4^2 - x_5^2
\end{align*}
Using the projection $\pi: \mathbb{R}^2 \times \mathbb{R}^3 \rightarrow \mathbb{R}^2$, we obtain a trisection near the Morse critical point by pulling back the rectangular trisection of the disk. In this model, the central submanifold is the hyperplane $\{x_1 = x_2 = 0\}$ and the restriction of $\widetilde{f}$ to the central submanifold is a Morse function with a critical point of index 1. By earlier discussion in this subsection, we can take this trisection to agree with $\T_X$.

Let $g$ be the standard Euclidean metric on $\mathbb{R}^5$ and $\nabla f$ the gradient of $f$ with respect to $g$.  The desending manifold of the critical point, with respect to $\nabla f$, is contained in the line $\{x_1 = x_2 = x_3 = x_4 = 0\}$ and intersects the level set $f^{-1}(-\epsilon)$ in the $0$-sphere $R = (0,0,0,0,\pm \sqrt{\epsilon})$.  Let $\widetilde{R} = (0,0,\pm \sqrt{\epsilon}) \subset \mathbb{R}^3$ be its projection.  Let $\nu(\widetilde{R}) \subset \mathbb{R}^3$ be a tubular neighborhood of $\widetilde{R}$ in $\widetilde{f}^{-1}(-\epsilon)$.  Flowing along $\nabla \widetilde{f}$, we obtain tubular neighborhoods of $(0,0,\pm( \sqrt{\epsilon + \delta})$ in $\widetilde{f}^{-1}(-\epsilon - \delta))$ for all $\delta > 0$.  We can find a tubular neighborhood $\nu(R)$ of $R$ in $f^{-1}(-\epsilon)$ of the form
\[\nu(R) \cong \nu(\widetilde{R}) \times D^2 = \nu(\widetilde{R}) \times g^{-1}([0,\epsilon/2])\]
The trisection of the local model restricts a trisection of $\nu(R)$ obtained by pulling back the rectangular trisection of the disk under the projection $\nu(R) \rightarrow D^2$.

Via an identification
\[ \nu_X(S) \cong \nu_{\Sigma}(S) \times D^2 \cong \nu(\widetilde{R}) \times D^2 \cong \nu(R)\]
we can use this model to extend a trisection from below the critical point to above the critical point. See Figure~\ref{fig:index1}.

\begin{figure}
\includegraphics[width=.5\textwidth]{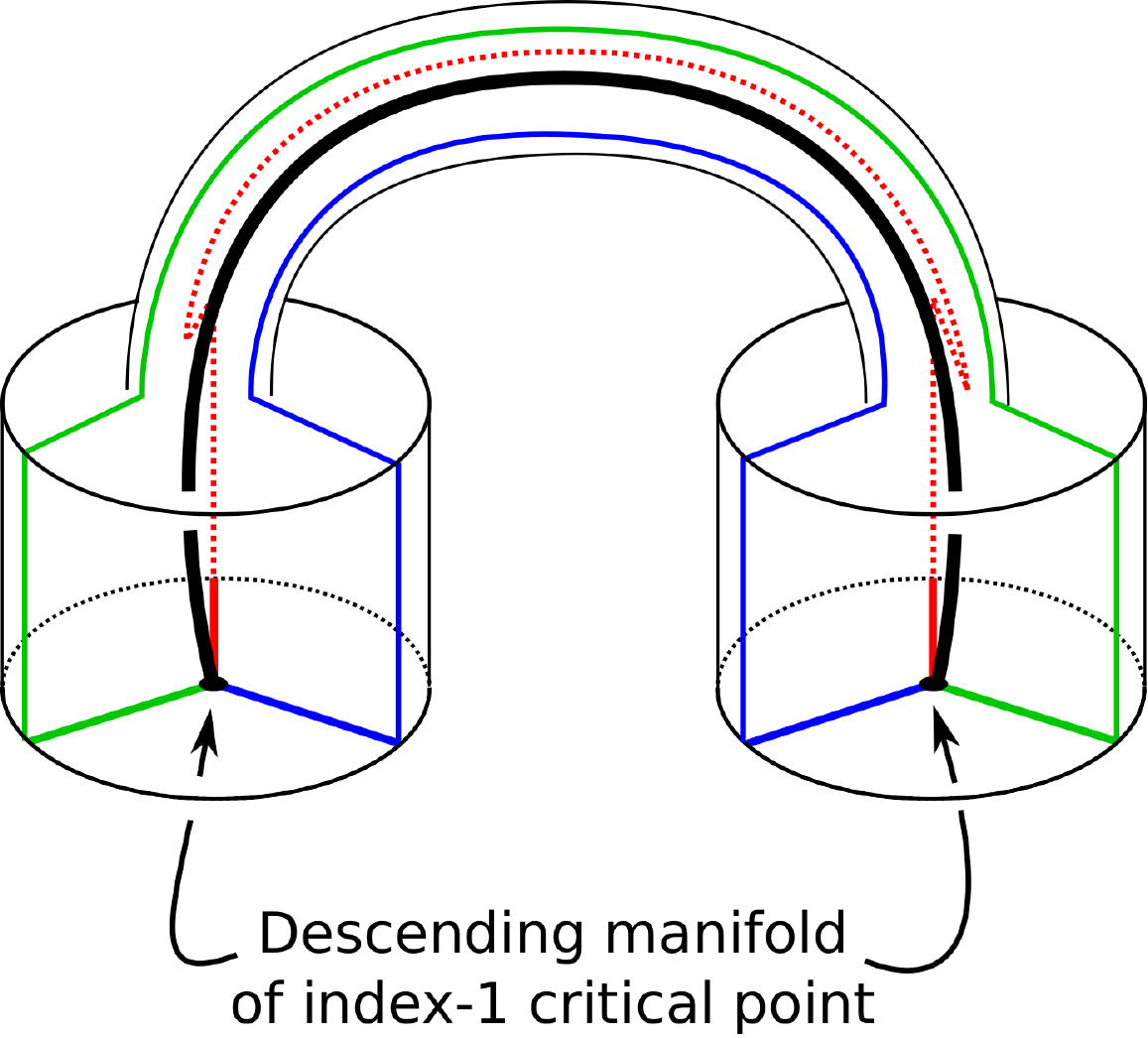}
\caption{The construction of Proposition~\ref{prop:index1}. We trisect a cobordism $Y$ from $X$ to $Z$ which includes an index-$1$ critical point. We isotope so that the descending manifold of the index-$1$ point intersects $X$ in the central surface of a trisection $\T_X$ on $X$. We then trisect the $1$-handle radially.}
\label{fig:index1}
\end{figure}

The topological result is as follows.  In the local model, The central submanifold is the hyperplane $\{x_1 = x_2 = 0\}$ and the function $f$ restricts to a Morse function with a critical point of index-$1$.  Thus, moving from height $-\epsilon$ to height $\epsilon$ results in surgery on $S \subset \Sigma$, increasing the genus by 1.  The double intersection $H_1 = D^r_1 \cap D^r_3$ is $\{x_1 \geq 0; x_2 = 0\}$ and the restriction of $f$ restricts has a boundary unstable critical point of index 1.  Thus, topologically moving from height $-\epsilon$ to height $\epsilon$ adds a 1-handle to $H_1$ along $S$.  By symmetry, this is also true of the remaining double intersections.  Finally, the top dimensional sector $D^r_1$, in the rectangular model of the trisection, is the halfspace $\{x_1 \geq 0\}$.  The restriction of $f$ has a boundary unstable critical point of index 1 and so the topological effect when moving from height $-\epsilon$ to $\epsilon$ is to add a 1-handle to $D^r_1$ along $S$.  Again by symmetry, this is true of the remaining sectors.

The local model gives a trisection $\M'$ on $Y$ that restricts to $\T_X$ on $X$ and some $\T'_Z = (Z'_1,Z'_2,Z'_3)$ on $Z$.  The sector $X_i$ is a 4-dimensional 1-handlebody of genus $k_i$.  The sector $Y_i$ is obtained by thickening $X_i$ and attaching a 1-handle.  Thus clearly the inclusion $X_i \hookrightarrow Y_i$ includes the cores of $X_i$ into the cores of $Y_i$.  Moreover, the sector $Y_i$ retracts onto $Z'_i$.  Thus $\M'$ is strongly compatible.  By Proposition \ref{prop:change-trisection}, we can find a trisection on $Z \times I$ strongly compatible with $\T'_Z$ and $\T_Z$.  Then by Lemma \ref{gluinglemma}, we can glue these together to obtain the required trisection.
\end{proof}

\subsection{Index-2}

\begin{lemma}\label{lemmaindex2}
Let $S$ be some embedded $S^1 \subset X$.  After stabilizing the trisection $\T_X$ and isotopy of $S$, we can assume that $S$ lies in the central surface $\Sigma$ of $\T_X$ so that it includes in each 3D and 4D piece as a core, and that the framing induced on $S$ by $\Sigma$ may be chosen arbitrarily.
\end{lemma}

\begin{proof}
Since $X_1\cap X_2$ generates the unbased fundamental group of $X$, we may isotope $S$ to lie in $X_1\cap X_2$, disjoint from a $1$-skeleton of $X_1\cap X_2$. Let $\pi(S)$ denote the projection of $S$ onto $\Sigma=\boundary(X_1\cap X_2)$. Take $\pi(S)$ to have only double points of self-intersection, and let $c(S)$ denote the number of such double points. Then we may perform $2+c(S)$ $3$-stabilizations (see Figure~\ref{fig:2handlecircle}) so that $S$ is a core in the interior of $X_3$. Perform a $1$- and  a $2$-stabilization, taking $S$ to run through the core of each added genus to $X_1$ and $X_2$, and project $S$ onto $\Sigma$ to obtain an embedded curve $C$ in $\Sigma$. By construction, this projection can be taken to be an isotopy.

Let $A$ be the $\alpha$ curve in a triple of $\alpha,\beta,\gamma$ curves arising from the $1$-stabilization, so $A$ is parallel to a $\beta$ curve, intersects a $\gamma$ curve in point, and intersects $S$ in one point. Since $A$ bounds a disk whose interior is disjoint from $\Sigma$, $\mu(A)=0$, where $\mu:H_1(\Sigma;\Z)\to\Z/2$ is the Rokhlin quadratic form. Let $C'$ be a curve in $\Sigma$ obtained by Dehn twisting $C$ about $A$ (in either direction). Then $\mu(C)\neq\mu(C')$. Note $C'$ is isotopic to $C$ in $X$, so $C'$ is isotopic to $S$. Both $C$ and $C'$ include into each 3D and 4D piece of $\T$ as a core. Because there are only two possible framings on $S\subset X$, one of $C$ or $C'$ is the desired curve.

\end{proof}

\begin{figure}
\centering
\includegraphics[width=.85\textwidth]{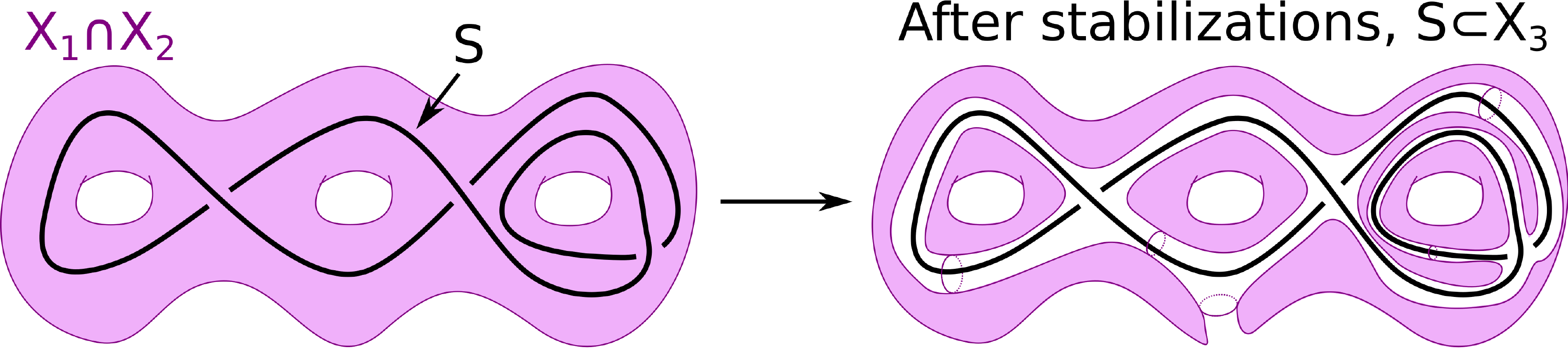}
\caption{In Lemma~\ref{lemmaindex2}, we isotope $S$ to lie in $X_1\cap X_2$ and then $3$-stabilize until $S$ is a core of $X_3$.}\label{fig:2handlecircle}
\end{figure}

\begin{proposition}
\label{prop:index2}
Suppose that there is a unique  critical point of $f$ of index-$2$ in the interior of $Y$. There exists a trisection $\M=(Y_1, Y_2, Y_3)$ of $Y$ that is strongly compatible with the trisections $\T_X$ and $\T_Z$.
\end{proposition}

\begin{proof}
The proof is analogous to the proof of Proposition~\ref{prop:index1}, so we only sketch it.  Near a Morse critical point of index 2, we have Morse coordinates such that
\[f = x_1^2 + x_2^2 + x_3^2 - x_4^2 - x_5^2\]
We can view this as a function on $\mathbb{R}^2 \times \mathbb{R}^3$ and decompose $f$ as $g + \widetilde{f}$, where
\begin{align*}
g(x_1,x_2) &= x_1^2 + x_2^2 & \widetilde{f}(x_3,x_4,x_5) &= x_3^2 - x_4^2 - x_5^2
\end{align*}
Using the projection $\pi: \mathbb{R}^2 \times \mathbb{R}^3 \rightarrow \mathbb{R}^2$, we obtain a trisection near the Morse critical point by pulling back the rectangular trisection of the disk.  In this model, the central submanifold is the hyperplane $\{x_1 = x_2 = 0\}$ and the restriction of $\widetilde{f}$ to the central submanifold is a Morse function with a critical point of index 2.  The desending manifold of the critical point, with respect to $\nabla f$, is the contained in the plane $\{x_1 = x_2 = x_3 = 0\}$ and intersects the level set $f^{-1}(-\epsilon)$ in the $1$-sphere $R$ and let $\widetilde{R} = (0,0,\pm \sqrt{\epsilon}) \subset \mathbb{R}^3$ be its projection.  Let $\nu(\widetilde{R}) \subset \mathbb{R}^3$ be a tubular neighborhood of $\widetilde{R}$ in $\widetilde{f}^{-1}(-\epsilon)$.  We can find a tubular neighborhood $\nu(R)$ of $R$ in $f^{-1}(-\epsilon)$ of the form
\[\nu(R) \cong \nu(\widetilde{R}) \times D^2 = \nu(\widetilde{R}) \times g^{-1}([0,\epsilon/2])\]
The trisection of the local model restricts a trisection of $\nu(R)$ obtained by pulling back the rectangular trisection of the disk under the projection $\nu(R) \rightarrow D^2$.  By Lemma~\ref{lemmaindex2}, we can take this trisection to agree with $\T_X$. Via an identification
\[ \nu_X(S) \cong \nu_{\Sigma}(S) \times D^2 \cong \nu(\widetilde{R}) \times D^2 \cong \nu(R)\]
we can use this model to extend a trisection from below the critical point to above the critical point.

The topological result is as follows.  In the local model, The central submanifold is the hyperplane $\{x_1 = x_2 = 0\}$ and the function $f$ restricts to a Morse function with a critical point of index 2.  Thus, moving from height $-\epsilon$ to height $\epsilon$ results in surgery on $S \subset \Sigma$, decreasing the genus by 1.  The double intersection $H_1 = D^r_1 \cap D^r_3$ is $\{x_1 \geq 0; x_2 = 0\}$ and the restriction of $f$ restricts has a boundary unstable critical point of index 2.  Thus, topologically this adds a 2-handle to $H_1$ along $S$; however, by assumption, this handle is attached along a core and therefore cancels a 1-handle in $H_1$.  By symmetry, this is also true of the remaining double intersections.  Finally, the top dimensional sector $D^r_1$, in the rectangular model of the trisection, is the halfspace $\{x_1 \geq 0\}$.  The restriction of $f$ has a boundary unstable critical point of index 2 and so the topological effect of moving from height $-\epsilon$ to $\epsilon$ is to add a 2-handle to $D^r_1$ along $S$.  Since $S$ represents a core, this 2-handle cancels a 1-handle.  Again by symmetry, this is true of the remaining sectors.

Strong compatibility follows as in the proof of Proposition \ref{prop:index1}. Then we can glue this local model to a model that changes the trisection on $Z$ to obtain the required trisection of $Y$.
\end{proof}

We remark that we now have all of the technology needed to prove Theorem \ref{thrm:cobord}. For now, if $f$ is self-indexing with only index-$1$, -$2$, -$3$, -$4$ points, we may trisect $f^{-1}[0,5/2]$ and $f^{-1}[5/2,5]$ separately (by turning $f^{-1}[5/2,5]$ upside down) and then glue the two copies of $f^{-1}(5/2)$. However, we instead continue to build the trisection upward from level $0$ for completion of the analogy between the handle structure and the trisection structure.

\subsection{Index-3}

\begin{lemma}\label{lemma1bridge}
Let $S$ be an embedded $S^2 \subset X$. 
By an isotopy, we can assume that $S$ is in 1-bridge position with respect to a stabilization of $\T_X$.
\end{lemma}

\begin{proof}
This is a specialization of~\cite[Theorem 1.2]{MZ}.
\end{proof}

\begin{proposition}
\label{prop:index3}
Suppose that there is a unique  critical point of $f$ of index-$3$ in the interior of $Y$. There exists a trisection $\M=(Y_1, Y_2, Y_3)$ of $Y$ that is strongly compatible with the trisections $\T_X$ and $\T_Z$.
\end{proposition}

\begin{proof}
This model can be obtained by turning the model in Proposition~\ref{prop:index1} upside-down.   Near a Morse critical point of index 3, we have Morse coordinates such that
\[f = x_1^2 + x_2^2 - x_3^2 - x_4^2 - x_5^2\]
We can view this as a function on $\mathbb{R}^3 \times \mathbb{R}^2$ and decompose $f$ as $\widetilde{f} - g$, where
\begin{align*}
\widetilde{f}(x_1,x_2,x_3) &= x_1^2 + x_2^2 - x_3^2 & g(x_4,x_5) &= x_4^2 + x_5^2 & 
\end{align*}
Using the projection $\pi: \mathbb{R}^3 \times \mathbb{R}^2 \rightarrow \mathbb{R}^2$, we obtain a trisection near the Morse critical point by pulling back the rectangular trisection of the disk.  In this model, the central submanifold is the hyperplane $\{x_4 = x_5 = 0\}$ and the restriction of $\widetilde{f}$ to the central submanifold is a Morse function with a critical point of index 1. 

Now, however, the descending manifold intersects $f^{-1}(-\epsilon)$ along the 2-sphere $R = \{-x_3^2 - x_4^2 - x_5^2 = -\epsilon\}$.  The trisection of the local model restricts to give the standard rectangular trisection of $R$. We can choose a neighborhood $\nu(R)$ and a splitting $\nu(R) = R \times D^2$ such that the trisection of local model, restricted to $\nu(R)$, is the same as the trisection obtained by pulling back the standard trisection of $S^2$ by the projection $\nu(R) \rightarrow R$.

Now suppose $S$ is in 1-bridge position via Lemma~\ref{lemma1bridge}.  This means that $\T_X$ restricted to $S$ is exactly the standard rectangular trisection of $S^2$ (up to isotopy).  We can choose an identification $\nu(S) = S \times D^2$ such that $\T_X$ restricted to $\nu(S)$ is exactly the trisection obtained by pulling back the standard rectangular trisection of $S^2$ by the projection $\nu(S) \rightarrow S$.  Via an identification
\[ \nu_X(S) \cong \nu_{\Sigma}(S) \times D^2 \cong \nu(\widetilde{R}) \times D^2 \cong \nu(R)\]
we can use this model to extend a trisection from below the critical point to above the critical point.

The topological result is as follows.  In the local model, The central submanifold is the hyperplane $\{x_4 = x_5 = 0\}$ and the function $f$ restricts to a Morse function with a critical point of index 1.  Thus, moving from height $-\epsilon$ to height $\epsilon$ results in surgery on $S \subset \Sigma$, increasing the genus by 1.  The double intersection $H_1 = D^r_1 \cap D^r_3$ is $\{x_1 \geq 0; x_2 = 0\}$ and the restriction of $f$ restricts has a boundary stable critical point of index 2.  Thus, moving from height $-\epsilon$ to height $\epsilon$ does not change the topology.  The top dimensional sector, in the rectangular model of the trisection, is the plane $\{x_5 \geq 0\}$.  This is a boundary stable critical point of index 3, so moving from height $-\epsilon$ to height $\epsilon$ does not change the topology.  

Strong compatibility follows as in the proof of Proposition \ref{prop:index1}. Then we can glue this local model to a model that changes the trisection on $Z$ to obtain the required trisection of $Y$.
\end{proof}

\subsection{Index-4}

An $S^3 \subset X$ is in {\it standard position} with respect to some trisection if $S \pitchfork \Sigma$ is a simple closed curve $c$ that bounds in all three handlebodies.  When an embedded $S^3$ is in standard position with respect to $\T$, the restriction of $\T$ is exactly the standard rectangular trisection of $S^3$ (up to isotopy).  Note that if the curve $c$ is separating, this is a locally a model for connected sum; if the curve is nonseparating, this is a model for an $S^1 \times S^3$ factor.

\begin{lemma}\label{lemmaindex4}
Let $S$ be an embedded $S^3 \subset X$.  By an isotopy and a stabilization of the trisection $\T_X$, we can assume that $S$ is in standard position with respect to $\T_X$.
\end{lemma}

\begin{proof}
Suppose that $S$ is separating.  Then $X$ decomposes as a connected sum $X = X_1 \#_S X_2$.  Choose trisections $\T_1$ and $\T_2$ of $X_1$ and $X_2$, respectively.  $X$ therefore admits a trisection $\T_S = \T_1 \# \T_2$ and $S$ is in standard position with respect to $\T_S$.  The trisections $\T$ and $\T_S$ admit a common stabilization $\widetilde{T}$.  Furthermore, 1-stabilization preserves the fact that $S$ is in standard position.

Now suppose $S$ is nonseparating.  Let $\gamma$ be a closed curve that intersects $S$ transversely in a single point.  Then $X$ decomposes as a connected sum along the boundary of the tubular neighborhood $\nu(S \cup \gamma)$ into $X' \# S^1 \times S^3$.  Let $\T'$ be a trisection of $X'$ and let $\T_{sphere}$ be the standard trisection of $S^1 \times S^3$.  The sphere $S$ is isotopic to $\{pt\} \times S^3$ and this sphere is in standard position with respect to $\T_{sphere}$.  The connected sum $\T' \# \T_{sphere}$ is a trisection of $X$.  Again, the trisections $\T$ and $\T' \# \T_{sphere}$ have a common stabilization $\widetilde{T}$ and $S$ is in standard position with respect to this trisection.
\end{proof}

\begin{proposition}
\label{prop:index4}
Suppose that there is a unique  critical point of $f$ of index-$4$ in the interior of $Y$. There exists a trisection $\M=(Y_1, Y_2, Y_3)$ of $Y$ that is strongly compatible with the trisections $\T_X$ and $\T_Z$.
\end{proposition}

\begin{proof}

This model can be obtained by turning the model in Proposition~\ref{prop:index1} upside-down.   Near a Morse critical point of index 4, we have Morse coordinates such that
\[f = x_1^2 - x_2^2 - x_3^2 - x_4^2 - x_5^2\]
We can view this as a function on $\mathbb{R}^3 \times \mathbb{R}^2$ and decompose $f$ as $\widetilde{f} - g$, where
\begin{align*}
\widetilde{f}(x_1,x_2,x_3) &= x_1^2 - x_2^2 - x_3^2 & g(x_4,x_5) &= x_4^2 + x_5^2 & 
\end{align*}
Using the projection $\pi: \mathbb{R}^3 \times \mathbb{R}^2 \rightarrow \mathbb{R}^2$, we obtain a trisection near the Morse critical point by pulling back the standard rectangular trisection of the disk.  In this model, the central submanifold is the hyperplane $\{x_4 = x_5 = 0\}$ and the restriction of $\widetilde{f}$ to the central submanifold is a Morse function with a critical point of index 2.

Now, however, the descending manifold intersects $f^{-1}(-\epsilon)$ along the 3-sphere $R = \{-x_2^2-x_3^2 - x_4^2 - x_5^2 = -\epsilon\}$.  The trisection of the local model restricts to give the standard rectangular trisection of $R$.  We can choose a neighborhood $\nu(R)$ and a splitting $\nu(R) = R \times D^1$ such that the trisection of local model, restricted to $\nu(R)$, is the same as the trisection obtained by pulling back the standard trisection of $S^3$ by the projection $\nu(R) \rightarrow R$.

Now suppose $S$ is in standard position (using Lemma~\ref{lemmaindex4}).  This means that $\T_X$ restricted to $S$ is exactly the standard trisection of $S^3$.  We can choose an identification $\nu(S) = S \times D^1$ such that $\T_X$ restricted to $\nu(S)$, is exactly the trisection obtained by pulling back the standard rectangular trisection of $S^3$ by the projection $\nu(S) \rightarrow S$. Via an identification
\[ \nu_X(S) \cong \nu_{\Sigma}(S) \times D^1 \cong \nu(\widetilde{R}) \times D^1 \cong \nu(R)\]
we can use this model to extend a trisection from below the critical point to above the critical point.

The topological result is as follows.  In the local model, The central submanifold is the hyperplane $\{x_4 = x_5 = 0\}$ and the function $f$ restricts to a Morse function with a critical point of index 2.  Thus, moving from height $-\epsilon$ to height $\epsilon$ results in surgery on $S \subset \Sigma$, decreasing the genus by 1.  The double intersection $H_1 = D^r_1 \cap D^r_3$ is $\{x_1 \geq 0; x_2 = 0\}$ and the restriction of $f$ restricts has a boundary stable critical point of index 3.  Thus, moving from height $-\epsilon$ to height $\epsilon$ does not change the topology.  The top dimensional sector, in the rectangular model of the trisection, is the plane $\{x_5 \geq 0\}$.  This is a boundary stable critical point of index 4, so moving from height $-\epsilon$ to height $\epsilon$ does not change the topology.  

Strong compatibility follows as in the proof of Proposition \ref{prop:index1}. Then we can glue this local model to a model that changes the trisection on $Z$ to obtain the required trisection of $Y$.
\end{proof}

The following is a slightly stronger statement of Theorem \ref{thrm:cobord}.

\begin{theorem}\label{maintheorem}
Let $Y$ be a cobordism between smooth, closed, connected $4$-manifolds $X$ and $Z$. Fix trisections $\T_X=(X_1,X_2,X_3)$ of $X$ and $\T_Z=(Z_1, Z_2, Z_3)$ of $Z$.

Let $h:Y\to 5$ be a relative Morse function so that $\boundary Y=h^{-1}(\{0,5\})$, where $X=h^{-1}(0)$ and $Z=h^{-1}(5)$. Assume that $h\vert_{\int{Y^4}}$ has no index-$0$ or -$5$ critical points  and that all index-$i$ critical points of $h\vert_{\int{Y^4}}$ lie in $h^{-1}(i)$ for $i=1,2,3,4$.

Then there exists a trisection $\M=(Y_1, Y_2, Y_3)$ of $Y$ so that $Y_i\cap h^{-1}(0)=X_i$ and $Y_i\cap h^{-1}(5)=Z_i$ for each $i=1,2,3$. Moreover, $\M$ is strongly compatible with $Y$. 
\end{theorem}

\begin{proof}
The theorem follows via induction on Propositions~\ref{prop:index1},~\ref{prop:index2},~\ref{prop:index3}, and~\ref{prop:index4}.
%
\end{proof}

The following corollary follows immediately from Theorem~\ref{maintheorem} and the definition of strongly compatible.

\begin{corollary}
Let $Y$ be a cobordism between smooth, closed, connected $4$-manifolds $A$ and $B$. Let $W$ be a cobordism between closed, connected $4$-manifolds $B$ and $C$. Fix trisections $\T_A,\T_B,\T_C$ of $M,N,$ and $X$ respectively. Let $\M_Y, \M_W$ be trisections of $Y$ and $W$ as in Theorem~\ref{maintheorem} so that $\M_Y$ restricts to $\T_A$ on $A$, $\M_W$ restricts to $\T_C$ on $C$, and both $\M_Y$ and $\M_W$ restrict to $\T_B$ on $B$. Moreover, $\M_Y$ and $\M_W$ are strongly compatible with $Y$ and $W$, respectively. Then we can glue $\M_Y$ and $\M_W$ to obtain a trisection $\M$ of the cobordism $Y\cup_B W$ from $A$ to $C$.

Note here that the identification of $B\subset \boundary Y$ and $B\subset\boundary W$ is induced by $\T_B$.
\end{corollary}

\bibliographystyle{alpha}
\nocite{*}
\bibliography{References}

\end{document}